\documentclass[11pt]{amsart}
\usepackage[margin=1in]{geometry}

\usepackage{amssymb}
\usepackage{amsthm}
\usepackage{amsmath}
\usepackage{mathrsfs}
\usepackage{amsbsy}
\usepackage{bm}
\usepackage{hyperref}
\usepackage{tikz}
\usepackage{array}
\usepackage{enumerate}
\usepackage{bbm}
\usepackage{comment}
\usepackage{mathtools}
\usepackage{makecell} 
\usepackage{colortbl}
\usepackage{xcolor}

\usepackage{tikz}
\usepackage{tikz-cd}

\definecolor{LightBlue}{rgb}{0,0.8,1} 
\definecolor{Turquoise}{rgb}{0,0.7,0.4}
\definecolor{NewBlue}{rgb}{0,0.3,0.8}
\definecolor{Pink}{RGB}{255,0,255}
\definecolor{BlueGreen}{RGB}{0,255,120}

\hypersetup{colorlinks=true, citecolor=LightBlue, linkcolor=Turquoise,urlcolor=Turquoise}

\usepackage{hhline}
\allowdisplaybreaks
\usepackage[noadjust]{cite}

\usepackage{caption}

\usepackage[noabbrev,capitalise,nameinlink]{cleveref}
\crefname{conjecture}{Conjecture}{Conjectures}

\newtheorem{theorem}{Theorem}[section]
\newtheorem{proposition}[theorem]{Proposition}

\newtheorem{conjecture}[theorem]{Conjecture}
\newtheorem{question}[theorem]{Question}
\newtheorem{problem}[theorem]{Problem}
\newtheorem{lemma}[theorem]{Lemma}

\theoremstyle{definition}
\newtheorem{definition}[theorem]{Definition}
\newtheorem{remark}[theorem]{Remark}
\newtheorem{example}[theorem]{Example}

\newcommand{\row}{\mathrm{Row}}

\newcommand{\Ech}{\operatorname{Ech}} 
\newcommand{\DD}{\mathcal{D}}
\newcommand{\UU}{\mathcal{U}}
\newcommand{\Pre}{\mathrm{Pre}}
\newcommand{\Suc}{\mathrm{Suc}} 
\newcommand{\CC}{\mathbb{C}} 
\newcommand{\rr}{\varrho}
\newcommand{\bb}{\delta} 
\newcommand{\Cov}{\mathrm{Cov}} 
\newcommand{\popdown}{\mathrm{Pop}} 
 
\newcommand{\rk}{\mathrm{rk}} 
\newcommand{\JJ}{\mathcal{J}}
\newcommand{\MM}{\mathcal{M}} 
\newcommand{\Odown}{\Upsilon}

\newcommand{\rank}{\mathrm{rank}} 
 
\newcommand{\jj}{\mathfrak{j}} 
\newcommand{\w}{\mathbf{w}} 
\newcommand{\lex}{\mathrm{lex}} 
\newcommand{\C}{\mathcal{C}}

\usepackage{todonotes}

\newcommand{\dfn}[1]{\textcolor{NewBlue}{\emph{#1}}}

\begin{document}

\title[]{Rowmotion and Echelonmotion}  
\subjclass[2010]{}

\author[]{Colin Defant}
\address[]{Department of Mathematics, Harvard University, Cambridge, MA 02138, USA}
\email{colindefant@gmail.com} 

\author[]{Yuhan Jiang}
\address[]{Department of Mathematics, Harvard University, Cambridge, MA 02138, USA}
\email{yhj1761@gmail.com}

\author[]{Rene Marczinzik}
\address[]{Mathematical Institute of the University of Bonn, Endenicher Allee 60, 53115 Bonn, Germany}
\email{marczire@math.uni-bonn.de}

\author[]{Adrien Segovia}
\address[]{Lacim, UQAM, Montréal, QC, H3C 3P8 Canada}
\email{adrien.segovia@gmail.com} 

\author[]{David E Speyer}
\address[]{Department of Mathematics, University of Michigan, 
Ann Arbor, MI, 48109, USA}
\email{speyer@umich.edu} 

\author[]{Hugh Thomas}
\address[]{Lacim, UQAM, Montréal, QC, H3C 3P8 Canada}
\email{thomas.hugh\_r@uqam.ca}

\author[]{Nathan Williams}
\address[]{Department of Mathematical Sciences, University of Texas at Dallas, Richardson, TX, 75080, USA}
\email{nathan.f.williams@gmail.com}

\begin{abstract}
Given a linear extension $\sigma$ of a finite poset $R$, we consider the permutation matrix indexing the Schubert cell containing the Cartan matrix of $R$ with respect to $\sigma$. This yields a bijection $\mathrm{Ech}_\sigma\colon R\to R$ that we call \emph{echelonmotion}; it is the inverse of the Coxeter permutation studied by Kl\'asz, Marczinzik, and Thomas. Those authors proved that echelonmotion agrees with rowmotion when $R$ is a distributive lattice. We generalize this result to semidistributive lattices. In addition, we prove that every trim lattice has a linear extension with respect to which echelonmotion agrees with rowmotion. We also show that echelonmotion on an Eulerian poset (with respect to any linear extension) is an involution. Finally, we initiate the study of \emph{echelon-independent posets}, which are posets for which echelonmotion is independent of the chosen linear extension. We prove that a lattice is echelon-independent if and only if it is semidistributive. Moreover, we show that echelon-independent connected posets are bounded and have semidistributive MacNeille completions. 
\end{abstract} 

\maketitle

\section{Introduction}

\subsection{Rowmotion} 
Let $Q$ be a finite poset, and consider the distributive lattice $J(Q)$ of lower order ideals of $Q$, partially ordered by containment. For $X\subseteq Q$, let \[\nabla_Q(X)=\{q\in Q: q\geq x \text{ for some }x\in X\}\] be the upper order ideal generated by $X$. \dfn{Rowmotion} is the bijection $\row_{J(Q)}\colon J(Q)\to J(Q)$ defined by 
\[\row_{J(Q)}(I)=Q\setminus\nabla_Q(\max(I)),\] where $\max(I)$ is the set of maximal elements of $I$. Birkhoff's representation theorem states that every finite distributive lattice is isomorphic to the lattice of lower order ideals of a finite poset, so one can define a bijective rowmotion operator on any distributive lattice.  The left-hand side of \cref{fig:distributive_rowmotion} shows the action of rowmotion on the lower order ideals of a $3$-element poset.  

Rowmotion has been discovered and rediscovered in several contexts, including matroid theory~\cite{deza1990loops} and quiver representation theory~\cite{IM,MTY24}.  It has been studied extensively in dynamical algebraic combinatorics~\cite{Panyushev,SW12,AST13,Roby,B19,TW19,HopkinsOPAC}. 

\subsection{Echelonmotion} We are interested in a new incarnation of rowmotion that Kl\'asz, Marczinzik, and Thomas recently discovered while studying the grade bijection for Auslander--Gorenstein algebras \cite{KMT25}. This new interpretation uses the Bruhat decomposition of the general linear group $\mathrm{GL}_n(\mathbb C)$. Let $B$ be the set of upper-triangular invertible complex matrices. Recall that 
\begin{equation}\label{eq:bruhat}\mathrm{GL}_n(\mathbb C)=\bigsqcup_{P\in \mathfrak S_n}BPB,\end{equation} where $\mathfrak S_n$ denotes the symmetric group on $n$ letters, which we identify with the set of $n\times n$ permutation matrices. (See \cite{OOV} for more details.) 

Let $R$ be an $n$-element poset, and let $[n]=\{1,\ldots,n\}$.  A \dfn{linear extension} of $R$ is a bijection $\sigma\colon R\to [n]$ such that $\sigma(x)\leq\sigma(y)$ for all $x,y\in R$ satisfying $x\leq y$. Fix a linear extension $\sigma$ of $R$. The \dfn{Cartan matrix} of $R$ with respect to $\sigma$ is the $n\times n$ matrix $W^{R,\sigma}$ whose entry in row $i$ and column $j$ is 
\[W^{R,\sigma}_{i,j}=\begin{cases}
    1 & \text{if }\sigma^{-1}(i)\geq \sigma^{-1}(j) \\
    0 & \text{if }\sigma^{-1}(i)\not\geq \sigma^{-1}(j). 
\end{cases}\]
Because $W^{R,\sigma}$ is lower-triangular with 1's on its diagonal, it is invertible---so by~\eqref{eq:bruhat}, there is a unique permutation matrix $P^{R,\sigma}\in\mathfrak S_n$ such that $W^{R,\sigma}\in BP^{R,\sigma}B$. Using $\sigma$, we can view $P^{R,\sigma}$ as a bijection from $R$ to itself.

\begin{definition}\label{def:echelonmotion} 
Let $R$ be a finite poset. Given a linear extension $\sigma$ of $R$, we define \dfn{echelonmotion} with respect to $\sigma$ to be the bijection $\Ech_\sigma\colon R\to R$ such that $\Ech_\sigma(x)=y$ if and only if $P^{R,\sigma}$ has a $1$ in row $\sigma(y)$ and column $\sigma(x)$.
\end{definition} 

We say a poset $R$ is \dfn{echelon-independent} if $\Ech_\sigma=\Ech_{\sigma'}$ for all linear extensions $\sigma,\sigma'$ of~$R$.  Kl\'asz, Marczinzik, and Thomas \cite{KMT25} proved that if $R=J(Q)$ is a distributive lattice, then $\Ech_\sigma=\row_R$ for any linear extension $\sigma$ of $R$.\footnote{They actually worked with the bijection $\Ech_\sigma^{-1}$, which they called the \dfn{Coxeter permutation}. This name comes from the fact that it corresponds to the permutation matrix appearing in the Bruhat decomposition of the \emph{Coxeter matrix} $-(W^{R,\sigma})^{-1}(W^{R,\sigma})^\top$. The definition of rowmotion used in \cite{KMT25} is the inverse of the one we use here.} In particular, this shows that distributive lattices are echelon-independent. 

\begin{figure}[htbp]
  \begin{center}
  \includegraphics[height=4.768cm]{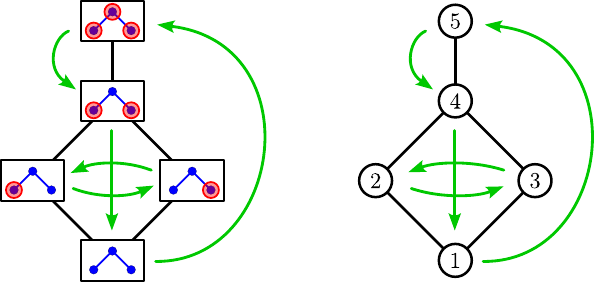}
  \end{center}
\caption{On the left is the distributive lattice of lower order ideals of a $3$-element poset, where each lower order ideal is represented as a collection of elements shaded in red. On the right is the same lattice labeled according to a linear extension. In each depiction of the lattice, a green arrow is drawn from an element to its image under rowmotion.  }\label{fig:distributive_rowmotion}
\end{figure}

\begin{example}\label{exam:distributive_rowmotion} 
Let $Q$ be the $3$-element poset $\begin{array}{l}\includegraphics[height=0.468cm]{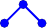}\end{array}$\!. The left side of \cref{fig:distributive_rowmotion} shows the lattice $R=J(Q)$, with the action of rowmotion represented by green arrows. The right side of the same figure shows the same distributive lattice with elements labeled by a linear extension $\sigma$. The Cartan matrix and the permutation matrix in its Bruhat decomposition are 
\[W^{R,\sigma}=\left(\begin{array}{ccccc}
1 & 0 & 0 & 0 & 0 \\
1 & 1 & 0 & 0 & 0 \\
1 & 0 & 1 & 0 & 0 \\ 
1 & 1 & 1 & 1 & 0 \\ 
1 & 1 & 1 & 1 & 1 \end{array} \right)\quad\text{and}\quad P^{R,\sigma}=\left(\begin{array}{ccccc}
0 & 0 & 0 & 1 & 0 \\
0 & 0 & 1 & 0 & 0 \\
0 & 1 & 0 & 0 & 0 \\ 
0 & 0 & 0 & 0 & 1 \\ 
1 & 0 & 0 & 0 & 0 \end{array}\right).\] Indeed, this follows from the factorization 
\[\left(\begin{array}{ccccc}
1 & 0 & 0 & 0 & 0 \\
1 & 1 & 0 & 0 & 0 \\
1 & 0 & 1 & 0 & 0 \\ 
1 & 1 & 1 & 1 & 0 \\ 
1 & 1 & 1 & 1 & 1 \end{array} \right)=\left(\begin{array}{ccccc}
1 & 1 & 1 & 1 & 1 \\
0 & 1 & 0 & 1 & 1 \\
0 & 0 & 1 & 1 & 1 \\ 
0 & 0 & 0 & 1 & 1 \\ 
0 & 0 & 0 & 0 & 1 \end{array}\right)\left(\begin{array}{ccccc}
0 & 0 & 0 & 1 & 0 \\
0 & 0 & 1 & 0 & 0 \\
0 & 1 & 0 & 0 & 0 \\ 
0 & 0 & 0 & 0 & 1 \\ 
1 & 0 & 0 & 0 & 0 \end{array}\right)\left(\begin{array}{ccccc}
1 & 1 & 1 & 1 & 1 \\
0 & -1 & 0 & -1 & 0 \\
0 & 0 & -1 & -1 & 0 \\ 
0 & 0 & 0 & 1 & 0 \\ 
0 & 0 & 0 & 0 & -1 \end{array}\right).\]
Note that the permutation matrix $P^{R,\sigma}$, viewed as a permutation of $R$, agrees with rowmotion. 
\end{example}

There has been interest in recent years in extending the definition of rowmotion to larger families of posets besides distributive lattices \cite{B19,DW23,TW19,TW20}. Echelonmotion provides such a generalization that is quite broad in the sense that it is defined for any finite poset with respect to any linear extension. Our aim is to study echelonmotion for several notable families of posets. 

\subsection{Main Results}

\emph{All posets in this article are assumed to be finite.}  

Semidistributive lattices generalize distributive lattices. They have a representation theorem that mimics Birkhoff's representation theorem for distributive lattices \cite{RST21}. Notable examples of semidistributive lattices include intervals in weak order on Coxeter groups \cite{BB,ReadingFiniteCoxeter}, the facial weak orders of simplicial hyperplane arrangements \cite{DHMP}, Cambrian lattices \cite{ReadingCambrian}, $\nu$-Tamari lattices \cite{nuTamariCombinatorica,PRV}, framing lattices \cite{Framing}, and lattices of torsion classes of finite-dimensional algebras \cite{DIRRT}. 
Barnard \cite{B19} found a natural definition of a rowmotion operator ${\row_L\colon L\to L}$ when $L$ is a semidistributive lattice (see \cref{sec:semidistributive} for the definition). We will use this definition to obtain the following generalization of the aforementioned result from~\cite{KMT25}. 

\begin{theorem}\label{semidist-thm} 
Let $L$ be a semidistributive lattice. We have \[\Ech_{\sigma}=\row_L\] for every linear extension $\sigma$ of $L$. In particular, $L$ is echelon-independent.  Furthermore, every echelon-independent lattice is semidistributive. 
\end{theorem} 

We next turn our attention to trim lattices, which were introduced by Thomas \cite{T06} as generalizations of distributive lattices that need not be graded. Notable examples of trim lattices include distributive lattices, Cambrian lattices \cite{ReadingCambrian}, $\nu$-Tamari lattices \cite{nuTamariCombinatorica,PRV}, and certain lattices of torsion classes of finite-dimensional algebras \cite[Corollary~1.5]{TW19}. Thomas and Williams \cite{TW19} defined a rowmotion operator $\row_L\colon L\to L$ whenever $L$ is trim. When $L$ is both trim and semidistributive, their definition of rowmotion agrees with Barnard's definition and, therefore, coincides with echelonmotion on $L$ (with respect to any linear extension) by \cref{semidist-thm}. When $L$ is trim but not semidistributive, \cref{semidist-thm} tells us that we cannot have $\Ech_\sigma=\row_L$ for \emph{every} linear extension $\sigma$ of $L$. Nonetheless, we will define the notion of a \emph{vertebral} linear extension of a trim lattice (see \cref{def:vertebral}) and prove the following theorem.  

\begin{theorem}\label{thm:trim} 
Let $L$ be a trim lattice. If $\sigma$ is a vertebral linear extension of $L$, then $\Ech_\sigma=\row_L$. 
\end{theorem} 

We also consider Eulerian posets, notable examples of which include face lattices of polytopes and intervals of Bruhat order on Coxeter groups. 

\begin{theorem}\label{thm:Eulerian} 
Let $R$ be an Eulerian poset. For every linear extension $\sigma$ of $R$, the map $\Ech_\sigma$ is an involution. 
\end{theorem}

Echelon-independent posets are especially nice because they come equipped with an intrinsic echelonmotion bijection (without the additional specification of a linear extension). While we do not have a characterization of such posets, we do have some necessary conditions and some sufficient conditions for echelon-independence. In \cref{thm:algorithm1,thm:algorithm2}, we provide an algorithm to test whether or not a poset is echelon-independent. This algorithm is useful in practice; for example, it has allowed us to verify that the Bruhat order on the symmetric group $S_n$ is echelon-independent whenever $n\leq 5$ and is not echelon-independent when $n=6$. We also have the following additional results. Recall that a poset is \dfn{bounded} if it has a minimum element and a maximum element. A poset is called \dfn{connected} if its Hasse diagram is a connected graph. 

\begin{theorem}\label{thm:bounded}
Every echelon-independent connected poset is bounded. 
\end{theorem} 

\begin{theorem}\label{thm:fixed}
Let $R$ be an echelon-independent connected poset of cardinality at least $2$. Then echelonmotion on $R$ has no fixed points. 
\end{theorem} 

\begin{theorem}\label{thm:MacNeille} 
The MacNeille completion of an echelon-independent connected poset is a semidistributive lattice. 
\end{theorem} 

In \cref{sec:echelon-independent}, we provide a counterexample to the converse of \cref{thm:MacNeille} by exhibiting a connected poset that is not echelon-independent but whose MacNeille completion is distributive (and hence, semidistributive).

\subsection{Outline} 
\cref{sec:background} provides background on posets, lattices, and Bruhat decomposition. In \cref{sec:basics}, we establish some useful results about echelonmotion. In \cref{sec:echelon-independent}, we prove \cref{thm:bounded,thm:fixed,thm:MacNeille}; we also prove \cref{thm:algorithm1,thm:algorithm2}, which produce an algorithm to test echelon-independence. \cref{sec:semidistributive} concerns semidistributive lattices and is devoted to proving \cref{semidist-thm}. \cref{sec:trim} concerns trim lattices and is devoted to proving \cref{thm:trim}. \cref{sec:Eulerian} concerns Eulerian posets and is devoted to proving \cref{thm:Eulerian}. Finally, \cref{sec:conclusion} presents several suggestions for future work.

\section{Background}\label{sec:background}  
Let $R$ be a poset with partial order relation $\leq$. A \dfn{subposet} of $R$ is a subset of $R$ equipped with the restriction of the partial order $\leq$. 
For $x,y\in R$, the \dfn{interval} from $x$ to $y$ is the set $[x,y]=\{z\in R:x\leq z\leq y\}$, which we view as a subposet of $R$. We say $y$ \dfn{covers} $x$ and write $x\lessdot y$ if $[x,y]=\{x,y\}$ and $x\neq y$. Let $\Cov_R^\downarrow(x)=\{y\in R:y\lessdot x\}$ be the set of elements covered by $x$, and let $\Cov_R^\uparrow(x)=\{y\in R:x\lessdot y\}$ be the set of elements that cover $x$. For $x\in R$, let 
\[\Delta_R(x)=\{y\in R:y\leq x\}\quad\text{and}\quad\nabla_R(x)=\{y\in R: y\geq x\}\]
be the principal lower order ideal and the principal upper order ideal generated by $x$. 

A \dfn{chain} in $R$ is a totally ordered subset of $R$. A chain is \dfn{maximal} if it is not a proper subset of a larger chain. The \dfn{order complex} of $R$ is the simplicial complex whose faces are the chains in $R$. 

We say $R$ is \dfn{bounded} if it has a minimum element $\hat 0$ and a maximum element $\hat 1$. In this case, the elements covering $\hat 0$ are called \dfn{atoms}, and the elements covered by $\hat 1$ are called \dfn{coatoms}. 

The \dfn{dual} of $R$, denoted $R^*$, is the poset with the same underlying set as $R$ but with the opposite partial order. That is, $x\leq y$ in $R^*$ if and only if $y\leq x$ in $R$. 

The \dfn{M\"obius function} of $R$ is the unique map $\mu_R\colon\{(x,y)\in R\times R:x\leq y\}\to\mathbb Z$ such that 
$\sum_{z\in [x,y]}\mu_R(x,z)=\delta_{x,y}$ for every interval $[x,y]$, where $\delta$ denotes the Kronecker delta. When $R$ has a minimum element $\hat 0$, we can restrict $\mu_R$ to a function on $R$ by letting $\mu_R(x)=\mu_R(\hat 0,x)$ for every $x\in R$. It is known \cite[Chapter~3]{Stanley} that $\mu_R(x,y)$ is equal to the reduced Euler characteristic of the order complex of the subposet $[x,y]\setminus\{x,y\}$. 

A \dfn{rank function} on $R$ is a map $\rk\colon R\to\mathbb Z$ such that $\rk(y)=\rk(x)+1$ for every cover relation $x\lessdot y$; if such a map exists, we say $R$ is \dfn{graded}. We say $R$ is \dfn{Eulerian} if it is graded and its M\"obius function satisfies $\mu_R(x,y)=(-1)^{\rk(y)-\rk(x)}$ for all $x,y\in R$ with $x\leq y$. 

A \dfn{lattice} is a poset $L$ such that any two elements $x,y\in L$ have a greatest lower bound, which is called their \dfn{meet} and denoted $x\wedge y$, and a least upper bound, with is called their \dfn{join} and denoted $x\vee y$. The meet and join operations on a lattice are commutative and associative, so we may consider the meet and join of any subset $X\subseteq L$, which we denote by $\bigwedge X$ and $\bigvee X$, respectively. 

Let $L$ be a lattice. Following \cite{DefantMeeting,DefantPopCoxeter}, we define the \dfn{pop-stack operator} $\popdown_L\colon L\to L$ by 
\[\popdown_L(x)=\bigwedge(\{x\}\cup\Cov_L^\downarrow(x)).\] 
Note that the presence of $\{ x \}$ in the definition is to handle the case where $\Cov_L^\downarrow(x)=\emptyset$; otherwise, the term $\{ x \}$ is redundant.
Let us also define 
\begin{equation}\label{eq:O}
\Odown_L(x)=\{z\in L:z\wedge x=\popdown_L(x)\}.
\end{equation}
An element of $L$ is \dfn{join-irreducible} (respectively, \dfn{meet-irreducible}) if it covers (respectively, is covered by) exactly $1$ element. Let $\JJ_L$ and $\MM_L$ denote the set of join-irreducible elements and the set of meet-irreducible elements of $L$, respectively. For $j\in\JJ_L$ and $m\in\MM_L$, we write $j_*$ for the unique element covered by $j$ and write $m^*$ for the unique element that covers $m$. 

We will need the following simple lemma. 

\begin{lemma}\label{lem:max_are_meet-irreducible} 
Let $x\lessdot y$ be a cover relation in a lattice $L$. Every maximal element of the set ${\{z\in L:z\wedge y=x\}}$ is meet-irreducible, and every minimal element of the set $\{z\in L:z\vee x=y\}$ is join-irreducible. 
\end{lemma}

\begin{proof}
We will prove only the first statement since the second statement follows from an analogous dual argument. Let $q$ be a maximal element of $\{z\in L:z\wedge y=x\}$, and suppose by way of contradiction that $q$ is not meet-irreducible. We know that $q$ is not the maximum element of $L$ since $y\not\leq q$. Thus, there are distinct elements $w,w'\in L$ that cover $q$. The maximality of $q$ ensures that $w\wedge y\neq x$ and $w'\wedge y\neq x$. Because $x\leq q<w$ and $x\leq q<w'$, this forces $w\wedge y$ and $w'\wedge y$ to be equal to $y$. It follows that $y\leq w$ and $y\leq w'$, so $y\leq w\wedge w'=q$. This is a contradiction. 
\end{proof} 

A lattice $L$ is called \dfn{meet-semidistributive} if for all $x,y\in L$ such that $x\leq y$, the set \[\{z\in L:z\wedge y=x\}\] has a maximum element. A lattice is \dfn{join-semidistributive} if its dual is meet-semidistributive. A lattice is \dfn{semidistributive} if it is both meet-semidistributive and join-semidistributive. 

\begin{proposition}[{\cite[Theorem~2.56]{Free}}]\label{SDCondition}
Let $L$ be a lattice. Then $L$ is meet-semidistributive if and only if $\Odown_L(j)$ has a maximum element for every $j\in\JJ_L$. 
\end{proposition}

Following \cite{Markowsky92}, we say a lattice $L$ is \dfn{extremal} if it has a maximum-length chain of cardinality $k+1$, where $k=|\JJ_L|=|\MM_L|$. An element $x\in L$ is called \dfn{left modular} if for all
$y,z\in L$ satisfying $y\leq z$, we have the equality $(y\vee x)\wedge z=y\vee(x\wedge z)$. A lattice is \dfn{left modular} if it
has a maximal chain of left modular elements. A lattice is \dfn{trim} if it is both extremal and left modular \cite{T06,TW19}. 

We will occasionally need to invoke results from the article \cite{DW23}, which concerns a family of lattices called \emph{semidistrim lattices}. We will not need to define this family here; the main fact we need is that semidistributive lattices and trim lattices are semidistrim. Suppose $L$ is semidistrim, and let $\hat 0$ and $\hat 1$ be the minimum and maximum elements of $L$, respectively. One result we will need is \cite[Corollary~8.2]{DW23}, which states that the order complex of the subposet $L\setminus\{\hat 0,\hat 1\}$ is either contractible or homotopy equivalent to a sphere. Moreover, this order complex is contractible if and only if $\hat 0\neq\popdown_L(\hat 1)$. In addition, \cite[Theorem~7.8]{DW23} states that intervals in semidistrim lattices are semidistrim. Together with the aforementioned relationship between M\"obius functions and reduced Euler characteristics, these facts imply the following lemma. 

\begin{lemma}\label{lem:Mobius_nonzero}
Let $L$ be a semidistrim lattice. For every $x\in L$, we have $\mu_L(\popdown_L(x),x)\in\{-1,1\}$. 
\end{lemma} 

The \dfn{MacNeille completion} of a poset $R$ is the smallest lattice that contains $R$ as a subposet; it is unique up to isomorphism. We refer to \cite{Schroder16} for more details, including an explicit construction of the MacNeille completion.

\section{Basics of Echelonmotion}\label{sec:basics}  
In this section, we prove some results that will allow us to compute echelonmotion in many circumstances. 

Let $\mathrm{GL}_n(\CC)$ be the set of invertible $n\times n$ complex matrices, and let $B$ be the set of upper-triangular matrices in $\mathrm{GL}_n(\CC)$. For every matrix $M\in\mathrm{GL}_n(\CC)$, there is a unique $n\times n$ permutation matrix $P$ such that $M\in BPB$. There is a standard method to compute $P$ using ranks of submatrices of $M$. For $i,j\in[n]$, let $M_{\geq i,\leq j}$ denote the lower-left submatrix of $M$ consisting of entries that belong to rows with indices at least $i$ and columns with indices at most $j$. Then $P_{i,j}=1$ if and only if 
\begin{equation}\label{eq:ranks} 
\rank(M_{\geq i+1,\leq j-1})=\rank(M_{\geq i+1,\leq j})=\rank(M_{\geq i,\leq j-1})=\rank(M_{\geq i,\leq j})-1.   
\end{equation} 
Indeed, this follows from the observation that multiplying a matrix by an invertible upper-triangular matrix (on the left or the right) cannot change the ranks of the lower-left submatrices. 

Let $\sigma$ be a linear extension of an $n$-element poset $R$. For $x\in R$, we define 
\[\Pre_\sigma(x)=\sigma^{-1}([1,\sigma(x)])\quad\text{and}\quad \Suc_\sigma(x)=\sigma^{-1}([\sigma(x),n]).\] If we view $\sigma$ as a total order on $R$, then $\Pre_\sigma(x)$ is the set of elements that weakly precede $x$ in $\sigma$, while $\Suc_\sigma(x)$ is the set of elements that weakly succeed $x$ in $\sigma$. For $Q\subseteq R$, we define $\max_{\sigma}(Q)$ to be the element $q \in Q$ such that $\sigma(q) = \max\{ \sigma(q') : q' \in Q \}$. 

\begin{proposition}\label{prop:labeling} 
Let $\sigma$ be a linear extension of a poset $R$, and let $x,y\in R$. We have $\Ech_\sigma(x)=y$ if and only if there exist maps $\rr\colon\Pre_\sigma(x)\to\CC$ and $\bb\colon\Suc_\sigma(y)\to\CC$ such that the following conditions hold:  
\begin{enumerate} 
\item[\emph{(i)}] $\rr(x)\neq 0$.  
\item[\emph{(ii)}] $\bb(y)\neq 0$. 
\item[\emph{(iii)}] $\sum_{w\in\Pre_\sigma(x)\cap\Delta_R(y)}\rr(w)\neq 0$. 
\item[\emph{(iv)}] $\sum_{w\in\Suc_\sigma(y)\cap\nabla_R(x)}\bb(w)\neq 0$. 
\item[\emph{(v)}] For every $u\in \Suc_\sigma(y)\setminus\{y\}$, we have $\sum_{w\in\Pre_\sigma(x)\cap\Delta_R(u)}\rr(w)=0$. 
\item[\emph{(vi)}] For every $v\in \Pre_\sigma(x)\setminus\{x\}$, we have $\sum_{w\in\Suc_\sigma(y)\cap\nabla_R(v)}\bb(w)=0$. 
\end{enumerate}
\end{proposition} 
\begin{proof} 
Let $W=W^{R,\sigma}$ be the Cartan matrix of $R$ with respect to $\sigma$, and let $P=P^{R,\sigma}$ be the permutation matrix in the Bruhat decomposition of $W$. Let $j=\sigma(x)$ and $i=\sigma(y)$. Our goal is to show that $P_{i,j}=1$ if and only if the six listed conditions hold. 

Conditions (i) and (v) together state that there is a linear relation among the columns of $W_{\geq i+1,\leq j}$ (the coefficient of column $\sigma(z)$ is $\rr(z)$) in which the column $j$ has a nonzero coefficient, so they are equivalent to the equality $\rank(W_{\geq i+1,\leq j-1})=\rank(W_{\geq i+1,\leq j})$. Condition (iii) then says that the corresponding coefficients do not form a linear relation among the columns of $W_{\geq i,\leq j}$, which implies that the dimension of the null space of $W_{\geq i,\leq j}$ is smaller than that of $W_{\geq i+1,\leq j}$. By the rank-nullity theorem, this implies that $\rank(W_{\geq i,\leq j})-1=\rank(W_{\geq i+1,\leq j})$. We conclude that conditions (i), (iii), and (v) together are equivalent to the statement that \[\rank(W_{\geq i+1,\leq j-1})=\rank(W_{\geq i+1,\leq j})=\rank(W_{\geq i,\leq j})-1.\] A similar argument (swapping the roles of rows and columns) shows that conditions (ii), (iv), and (vi)  are equivalent to the statement that \[\rank(W_{\geq i+1,\leq j-1})=\rank(W_{\geq i,\leq j-1})=\rank(W_{\geq i,\leq j})-1.\] According to \eqref{eq:ranks}, this proves the desired result. 
\end{proof} 

\begin{proposition}\label{prop:new}
Let $\sigma$ be a linear extension of a poset $R$, and let $x,y\in R$. Suppose there exists a map $\rr\colon\Pre_\sigma(x)\to\CC$ such that $\rr(x)\neq 0$ and such that $\sum_{w\in\Pre_\sigma(x)\cap\Delta_R(u)}\rr(w)=0$ for every $u\in \Suc_\sigma(y)\setminus\{y\}$. Then $\sigma(\Ech_\sigma(x))\leq\sigma(y)$. 
\end{proposition} 
\begin{proof}
Let $W=W^{R,\sigma}$ be the Cartan matrix of $R$ with respect to $\sigma$, and let $P=P^{R,\sigma}$ be the permutation matrix in the Bruhat decomposition of $W$. Let $n=|R|$, and let $j=\sigma(x)$ and $i=\sigma(y)$. Fix $i'\in[n]$ such that $i'>i$; we will show that $\sigma(\Ech_\sigma(x))\neq i'$. The hypotheses ensure that there is a linear relation among the columns of $W_{\geq i+1,\leq j}$ (the coefficient of column $\sigma(z)$ is $\rr(z)$) in which the column $j$ has a nonzero coefficient. Hence, there is a linear relation among the columns of $W_{\geq i',\leq j}$ in which the column $j$ has a nonzero coefficient. This implies that $\rank(W_{\geq i',\leq j-1})=\rank(W_{\geq i',\leq j})$, so it follows from \eqref{eq:ranks} that $P_{i',j}=0$. Hence, $\sigma(\Ech_\sigma(x))\neq i'$. 
\end{proof} 

\begin{proposition}\label{cor:alpha} 
Let $\sigma$ be a linear extension of a lattice $L$. Suppose that $\mu_L(\popdown_L(x),x)\neq 0$ for every $x\in L$. If the map $x \mapsto \max_{\sigma}(\Odown_L(x))$ is a bijection from $L$ to $L$, then $\Ech_\sigma(x) = \max_{\sigma}(\Odown_L(x))$ for every $x\in L$. 
\end{proposition} 
\begin{proof} 
For convenience, let $\alpha(x)=\max_\sigma(\Odown_L(x))$. Assume $\alpha\colon L\to L$ is a bijection. We will use \cref{prop:new} to show that \begin{equation}\label{eq:alpha}
\sigma(\Ech_\sigma(x))\leq\sigma(\alpha(x))\end{equation} for every $x\in L$.

Let $x\in L$, and let $y=\alpha(x)$. Consider the interval
$Q_x=[\popdown_L(x),x]$ of $L$. Define a map $\rr\colon\Pre_\sigma(x)\to\CC$ by 
\[\rr(w)=\begin{cases}
    \mu_{Q_x}(w) & \text{if }w\in\Pre_\sigma(x)\cap Q_x \\
    0 & \text{if }w\in\Pre_\sigma(x)\setminus Q_x. 
\end{cases}\] 
Note that $\rr(x)=\mu_{Q_x}(x)=\mu_L(\popdown_L(x),x)\neq 0$. Now consider an element $u\in\Suc_\sigma(y)\setminus\{y\}$. The hypothesis that $y=\max_{\sigma}(\Odown_L(x))$ ensures that $u\not\in\Odown_L(x)$, so $u\wedge x\neq\popdown_L(x)$. We have \[\sum_{w\in\Pre_\sigma(x)\cap\Delta_L(u)}\rr(w)=\sum_{w\in \Delta_L(u)\cap Q_x}\mu_{Q_x}(w),
\]
and we wish to show that this sum is $0$. This is obvious if $\Delta_L(u)\cap Q_x=\emptyset$, so assume $\Delta_L(u)\cap Q_x\neq\emptyset$. 
This implies that $\popdown_L(x)\leq u$, so $u\wedge x\in Q_x$. Hence, $\Delta_L(u)\cap Q_x$ is the principal lower order ideal $\Delta_{Q_x}(u\wedge x)$ of $Q_x$. Since $u\wedge x\neq\popdown_L(x)$, it follows from the definition of the M\"obius function that 
\[\sum_{w\in \Delta_L(u)\cap Q_x}\mu_{Q_x}(w)=\sum_{w\in\Delta_{Q_x}(u\wedge v)}\mu_{Q_x}(w)=0.\] It follows from \cref{prop:new} that $\sigma(\Ech_\sigma(x))\leq\sigma(y)$. This proves \eqref{eq:alpha}. 

We will now prove by induction on $\sigma(\alpha(x))$ that $\Ech_\sigma(x)=\alpha(x)$. Fix $x\in L$, and let \[{Z=\{z\in L:\sigma(\alpha(z))<\sigma(\alpha(x))\}};\] we may assume by induction that $\Ech_\sigma(z)=\alpha(z)$ for every $z\in Z$. Because $\alpha$ is a bijection, this assumption implies that \[\sigma(\Ech_\sigma(Z))=\sigma(\alpha(Z))=[\sigma(\alpha(x))-1].\] Since $x\not\in Z$, we deduce that  $\sigma(\Ech_\sigma(x))\not\in[\sigma(\alpha(x))-1]$. By \eqref{eq:alpha}, we have $\sigma(\Ech_\sigma(x))=\sigma(\alpha(x))$, so $\Ech_\sigma(x)=\alpha(x)$. 
\end{proof} 

\begin{remark}
    The condition that $\mu_L(\popdown_L(x),x)\neq 0$ is nontrivial. For example, let $L$ be the lattice shown in \cref{fig:remark}. Then $\popdown_L(\hat{1}) = \hat{0}$, but $\mu(\hat{0}, \hat{1})=0$. 
\end{remark} 

\begin{figure}[htbp]
  \begin{center}
  \includegraphics[height=1.925cm]{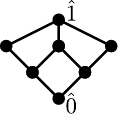}
  \end{center}
\caption{A lattice $L$ such that $\popdown_L(\hat{1}) = \hat{0}$ and $\mu(\hat{0}, \hat{1})=0$.  }\label{fig:remark}
\end{figure}

\section{Echelon-Independent Posets}\label{sec:echelon-independent} 

In this section, we prove \cref{thm:bounded,thm:fixed,thm:MacNeille}, which establish useful properties of echelon-independent posets. We also prove \cref{thm:algorithm1,thm:algorithm2}, which yield an efficient algorithm to test whether or not a poset is echelon-independent. 

\begin{lemma}\label{lem:min-to-max}
Let $R$ be a poset of size $n$. Let $x$ be a minimal element of $R$, and let $y$ be a maximal element of $R$ such that $x\leq y$. There exists a linear extension $\sigma$ of $R$ such that $\sigma(x)=1$ and $\sigma(y)=n$. Moreover, for any such linear extension $\sigma$, we have $\Ech_\sigma(x)=y$. 
\end{lemma}
\begin{proof}
The existence of $\sigma$ is straightforward. We have $\Pre_\sigma(x)=\{x\}$, and $\Suc_\sigma(y)=\{y\}$. Define $\rr\colon\Pre_\sigma(x)\to\CC$ and $\bb\colon\Suc_\sigma(y)\to\CC$ by letting $\rr(x)=\bb(y)=1$. It is straightforward to check that the six conditions in \cref{prop:labeling} are satisfied, so $\Ech_\sigma(x)=y$. 
\end{proof}

\begin{proof}[Proof of~\cref{thm:bounded}]
Let $R$ be a connected poset that is not bounded; we will show that $R$ is not echelon-independent. We will assume that $R$ has more than one maximal element; a similar argument handles the case in which $R$ has more than one minimal element. 

We claim that $R$ has a minimal element $x$ and two distinct maximal elements $y$ and $y'$ with $y > x < y'$. The hypothesis that $R$ is connected means that, for any maximal elements $y$ and $y'$ in $R$, there is a sequence $y=r_0 > r_1 < r_2 > \cdots < r_{2 \ell} = y'$. Take the shortest such sequence for which $y$ and $y'$ are distinct maximal elements. Then $\ell\geq 1$. If $\ell>1$, then we can choose some maximal $y''$ lying above $r_2$, and then either $y > r_1 < y''$ or $y'' > r_3 < r_4 > \cdots < r_{2 \ell} = y'$ is a shorter sequence with distinct endpoints. This is impossible, so $\ell=1$. Now, take $x$ to be a minimal element that is less than or equal to $r_1$.

According to \cref{lem:min-to-max}, there exist linear extensions $\sigma$ and $\sigma'$ of $R$ such that $\Ech_\sigma(x)=y$ and $\Ech_{\sigma'}(x)=y'$. Hence, $R$ is not echelon-independent. 
\end{proof} 

\begin{proof}[Proof of~\cref{thm:fixed}] 
Let $R$ be an echelon-independent connected poset of cardinality at least $2$. We know by \cref{thm:bounded} that $R$ is bounded; let $\hat 0$ and $\hat 1$ be the minimum and maximum elements of $R$, respectively. Suppose by way of contradiction that $x$ is a fixed point of echelonmotion on $R$. 

Let $n=|R|$. There exists a linear extension $\sigma$ of $R$ such that $\Suc_\sigma(x)=\nabla_R(x)$. We must have $\sigma(\hat 0)=1$ and $\sigma(\hat 1)=n$, so it follows from \cref{lem:min-to-max} that $\Ech_\sigma(\hat 0)=\hat 1\neq\hat 0$. Therefore, $x\neq\hat 0$. According to \cref{prop:labeling}, there exist maps $\rr\colon\Pre_\sigma(x)\to\CC$ and $\bb\colon\Suc_\sigma(x)\to\CC$ satisfying the six conditions listed in the statement of that proposition (with $y=x$). Condition (iv) tells us that $\sum_{w\in\nabla_R(x)}\bb(w)\neq 0$. Setting $v=\hat 0$ in condition (vi), we find that $\sum_{w\in \nabla_R(x)\cap\nabla_R(\hat 0)}\bb(w)=0$. This is a contradiction because $\nabla_R(x)\cap \nabla_R(\hat 0)=\nabla_R(x)$.  
\end{proof}

\begin{proof}[Proof of \cref{thm:MacNeille}] 
Let $R$ be a connected poset, and let $L$ be the MacNeille completion of $R$. Suppose $L$ is not semidistributive; we wish to show that $R$ is not echelon-independent. If $R$ is not bounded, then we already know that it is not echelon-independent by \cref{thm:bounded}. Hence, we will assume $R$ is bounded. We will assume $L$ is not meet-semidistributive; an analogous dual argument handles the case in which $L$ is not join-semidistributive. According to \cite[Proposition~8.26]{Schroder16}, every element of $L$ can be written both as a join of elements of $R$ and as a meet of elements of $R$. This implies that $\JJ_L$ and $\MM_L$ are contained in $R$. 

Because $L$ is not meet-semidistributive, it follows from \cref{SDCondition} that there exists a join-irreducible element $j\in\JJ_L$ such that the set $\Odown_L(j)=\{z\in L:z\wedge j=j_*\}$ has at least two maximal elements, say $m$ and $m'$. We know by \cref{lem:max_are_meet-irreducible} that $m,m'\in\MM_L$. We will construct a linear extension $\sigma$ of $R$ such that $\Ech_\sigma(j)=m$. The same argument will show that there exists a linear extension $\sigma'$ of $R$ such that $\Ech_{\sigma'}(j)=m'$, so it will follow that $R$ is not echelon-independent. 

Let $\hat 0$ and $\hat 1$ be the minimum of $R$ and the maximum of $R$, respectively. Because the minimum element of $L$ is a meet of elements of $R$, it must be equal to $\hat 0$. Similarly, the maximum element of $L$ is a join of elements of $R$, so it must be equal to $\hat 1$. This implies that $j\neq\hat 0$ and $m\neq \hat 1$. Thus, there exists an element $p$ of $R$ that is covered by $j$ in $R$, and there exists an element $q$ of $R$ that covers $m$ in $R$. 

Let $\sigma$ be a linear extension of $R$ such that $\Pre_\sigma(j)=\Delta_R(j)$ and $\Suc_\sigma(m)=\nabla_R(m)$. Define $\rr\colon\Pre_\sigma(j)\to\CC$ by 
\[\rr(w)=\begin{cases}
    1 & \text{if }w=j \\
    -1 & \text{if }w=p \\ 
    0 & \text{otherwise}.
\end{cases}\] 
Define $\bb\colon\Suc_\sigma(m)\to\CC$ by 
\[\bb(w)=\begin{cases}
    1 & \text{if }w=m \\
    -1 & \text{if }w=q \\ 
    0 & \text{otherwise}.
\end{cases}\] 
To complete the proof, we will show that these maps satisfy the conditions in \cref{prop:labeling} (with $x=j$ and $y=m$). Conditions (i) and (ii) are immediate. 

We have $p\leq j_*\leq m$ and $j\not\leq m$ (since $m\wedge j=j_*$). Therefore, we have $p\leq m$ and $j\not\leq m$ in $R$. It follows that \[\sum_{w\in\Pre_\sigma(j)\cap\Delta_R(m)}\rr(w)=\rr(p)=-1,\] which proves condition (iii). An analogous dual argument established condition (iv). 

Let us now prove condition (v); we will omit the proof of condition (vi) since it is completely analogous. Suppose $u\in\Suc_\sigma(m)\setminus\{m\}$. By the definition of $\sigma$, we have $m<u$. Since $m$ is a maximal element of the set $\{z\in L:z\wedge j=j_*\}$, we know that $u\wedge j\neq j_*$. But $j_*\leq m<u$, so we must have $u\wedge j=j$. This shows that $j\in\Delta_R(u)$. Consequently, 
\[\sum_{w\in\Pre_\sigma(j)\cap\Delta_R(u)}\rr(w)=\rr(j)+\rr(p)=0,\] as desired. 
\end{proof} 

The converse of \cref{thm:MacNeille} is false. \cref{fig:counterexample1} shows a poset labeled by two different linear extensions $\sigma$ and $\sigma'$ such that $\Ech_\sigma\neq\Ech_{\sigma'}$. However, the MacNeille completion of this poset is a Boolean lattice of cardinality $16$, which is distributive (and therefore, semidistributive).  

\begin{figure}[htbp]
  \begin{center}
  \includegraphics[height=6.419cm]{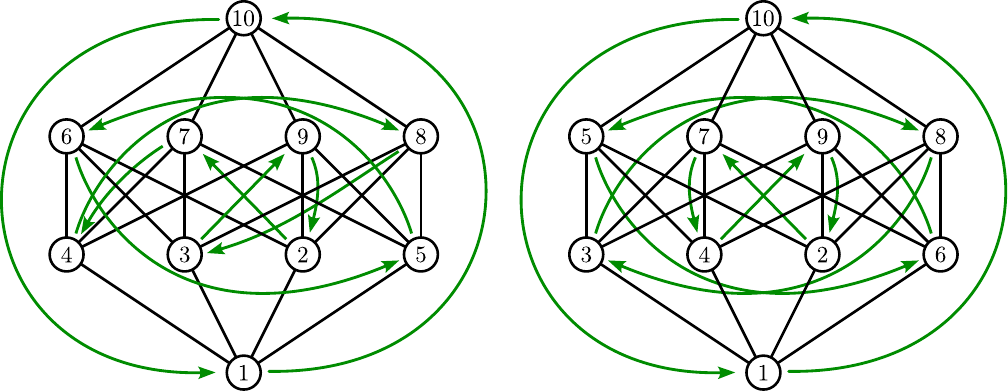} 
  \end{center}
\caption{Two linear extensions of a poset whose MacNeille completion is distributive. Echelonmotion with respect to each linear extension is represented by green arrows; note that the two echelonmotion maps are different.   }\label{fig:counterexample1} 
\end{figure}

We now provide an algorithm to test whether a poset is echelon-independent. Let us first establish a bit of notation. Let $\mathcal L(R)$ be the set of linear extensions of a poset $R$. Suppose $x$ and $y$ are elements of a poset $R$. If $x$ and $y$ are comparable, define 
\[\Lambda_1(x,y)=\{\lambda\in\mathcal L(R):\Pre_\lambda(x)=\Delta_R(x)\,\,\,\text{and}\,\,\,\Pre_\lambda(y)=\Delta_R(y)\},\] and 
\[\Lambda_2(x,y)=\{\lambda\in\mathcal L(R):\Suc_\lambda(x)=\nabla_R(x)\,\,\,\text{and}\,\,\,\Suc_\lambda(y)=\nabla_R(y)\}.\] If $x$ and $y$ are incomparable, let 
\begin{align*}
\Xi_1(x,y)&=\{\xi\in\mathcal L(R):\Pre_\xi(x)=\Delta_R(x)\,\,\,\text{and}\,\,\,\Pre_\xi(y)=\Delta_R(x)\cup\Delta_R(y)\}, \\ 
\Xi_2(x,y)&=\{\xi\in\mathcal L(R):\Pre_\xi(x)=\Delta_R(x)\cup\Delta_R(y)\,\,\,\text{and}\,\,\,\Pre_\xi(y)=\Delta_R(y)\}, \\  
\Xi_3(x,y)&=\{\xi\in\mathcal L(R):\Suc_\xi(x)=\nabla_R(x)\,\,\,\text{and}\,\,\,\Suc_\xi(y)=\nabla_R(x)\cup\nabla_R(y)\}, \\ 
\Xi_4(x,y)&=\{\xi\in\mathcal L(R):\Suc_\xi(x)=\nabla_R(x)\cup\nabla_R(y)\,\,\,\text{and}\,\,\,\Suc_\xi(y)=\nabla_R(y)\}.
\end{align*} 

\begin{proposition}\label{thm:algorithm1} 
Let $\sigma^\#$ be a linear extension of a poset $R$. Fix $x\in R$, and let $y=\Ech_{\sigma^\#}(x)$. Suppose $x$ and $y$ are comparable. Fix $\lambda_1\in\Lambda_1(x,y)$ and $\lambda_2\in\Lambda_2(x,y)$. If $\Ech_{\lambda_1}(x)=\Ech_{\lambda_2}(x)=y$, then $\Ech_\sigma(x)=y$ for every linear extension $\sigma$ of $R$. 
\end{proposition}  
\begin{proof}
Let $\sigma$ be a linear extension of $R$. Because $\Ech_{\lambda_1}(x)=y$, we know by \cref{prop:labeling} that there exists a map $\widetilde\rr\colon\Pre_{\lambda_1}(x)\to\CC$ satisfying $\widetilde\rr(x)\neq 0$ and $\sum_{w\in\Pre_{\lambda_1}(x)\cap\Delta_R(y)}\widetilde\rr(w)\neq 0$ and satisfying $\sum_{w\in\Pre_{\lambda_1}(x)\cap\Delta_R(u)}\widetilde\rr(w) =0$ for every $u\in\Suc_{\lambda_1}(y)\setminus\{y\}$. Because $\lambda_1\in\Lambda_1(x,y)$, we have \[\Pre_{\lambda_1}(x)=\Delta_R(x)\subseteq\Pre_\sigma(x)\quad\text{and}\quad \Suc_{\lambda_1}(y)\setminus\{y\}=R\setminus\Delta_R(y)\supseteq\Suc_\sigma(y)\setminus\{y\}.\] Define a map $\rr\colon\Pre_\sigma(x)\to\CC$ so that $\rr(w)=\widetilde\rr(w)$ for all $w\in\Delta_R(x)$ and $\rr(w)=0$ for all ${w\in\Pre_\sigma(x)\setminus\Delta_R(x)}$. It is straightforward to verify that $\rr$ satisfies conditions (i), (iii), and (v) from \cref{prop:labeling}. 

Because $\Ech_{\lambda_2}(x)=y$, we know by \cref{prop:labeling} that there exists a map $\widetilde\bb\colon\Suc_{\lambda_2}(y)\to\CC$ satisfying $\widetilde\bb(y)\neq 0$ and $\sum_{w\in\Suc_{\lambda_2}(y)\cap\nabla_R(x)}\widetilde\bb(w)\neq 0$ and satisfying $\sum_{w\in\Suc_{\lambda_2}(y)\cap\nabla_R(v)}\widetilde\bb(w)=0$ for every $v\in\Pre_{\lambda_2}(x)\setminus\{x\}$. Because $\lambda_2\in\Lambda_2(x,y)$, we have 
\[\Suc_{\lambda_2}(y)=\nabla_R(y)\subseteq\Suc_\sigma(y)\quad\text{and}\quad \Pre_{\lambda_2}(x)\setminus\{x\}=R\setminus\nabla_R(x)\supseteq\Pre_\sigma(x)\setminus\{x\}.\] Define a map $\bb\colon\Suc_\sigma(y)\to\CC$ so that $\bb(w)=\widetilde\bb(w)$ for all $w\in\nabla_R(y)$ and $\bb(w)=0$ for all ${w\in\Suc_\sigma(y)\setminus\nabla_R(y)}$. It is straightforward to verify that $\bb$ satisfies conditions (ii), (iv), and (vi) from \cref{prop:labeling}. We deduce from that proposition that $\Ech_\sigma(x)=y$. 
\end{proof} 

\begin{proposition}\label{thm:algorithm2} 
Let $\sigma^\#$ be a linear extension of a poset $R$. Fix $x\in R$, and let $y=\Ech_{\sigma^\#}(x)$. Suppose $x$ and $y$ are incomparable. For each $k\in\{1,2,3,4\}$, fix a linear extension $\xi_k\in\Xi_k(x,y)$. If $\Ech_{\xi_1}(x)=\Ech_{\xi_2}(x)=\Ech_{\xi_3}(x)=\Ech_{\xi_4}(x)=y$, then $\Ech_\sigma(x)=y$ for every linear extension $\sigma$ of~$R$. 
\end{proposition}  
\begin{proof}
Let $\sigma$ be a linear extension of $R$; we will show that $\Ech_\sigma(x)=y$. There are two cases to consider depending on whether $\sigma(x)<\sigma(y)$ or $\sigma(y)<\sigma(x)$. We will assume that $\sigma(x)<\sigma(y)$ and use the linear extensions $\xi_1$ and $\xi_4$ to complete the proof. (The argument in the other case is very similar, but it uses $\xi_2$ and $\xi_3$ instead of $\xi_1$ and $\xi_4$.) 

Note that 
\[\Pre_{\xi_1}(x)=\Delta_R(x)\subseteq\Pre_\sigma(x)\quad\text{and}\quad\Suc_{\xi_1}(y)\setminus\{y\}=R\setminus(\Delta_R(x)\cup\Delta_R(y))\supseteq\Suc_\sigma(y)\setminus\{y\}.\] This allows us to construct a map $\rr\colon\Pre_\sigma(x)\to\CC$ satisfying conditions (i), (iii), and (v) from \cref{prop:labeling}; the construction is essentially the same as in the proof of \cref{thm:algorithm1} (except we use $\xi_1$ instead of $\lambda_1$).  
Next, note that 
\[\Suc_{\xi_4}(y)=\nabla_R(y)\subseteq\Suc_\sigma(y)\quad\text{and}\quad\Pre_{\xi_4}(x)\setminus\{x\}=R\setminus(\nabla_R(x)\cup\nabla_R(y))\supseteq\Pre_\sigma(x)\setminus\{x\}.\]
This allows us to construct a map $\bb\colon\Suc_\sigma(y)\to\CC$ satisfying conditions (ii), (iv), and (vi) from \cref{prop:labeling}. We deduce from that proposition that $\Ech_\sigma(x)=y$. 
\end{proof} 

To use \cref{thm:algorithm1,thm:algorithm2} in practice to test whether or not an $n$-element poset $R$ is echelon-independent, one first chooses an arbitrary linear extension $\sigma^\#$ and computes $\Ech_{\sigma^\#}$. For each $x\in R$ such that $x$ and $\Ech_{\sigma^\#}(x)$ are comparable, one can use \cref{thm:algorithm1} to test whether or not $\Ech_\sigma(x)=\Ech_{\sigma^\#}(x)$ for every linear extension $\sigma$ of $R$; this requires one to test whether $\Ech_{\lambda_1}(x)=\Ech_{\lambda_2}(x)=y$ for just two specific linear extensions $\lambda_1$ and $\lambda_2$. By \eqref{eq:ranks}, this amounts to computing the ranks of $8$ matrices ($4$ for each linear extension), each of which has at most $n$ rows and at most $n$ columns. For each $x\in R$ such that $x$ and $\Ech_{\sigma^\#}(x)$ are not comparable, one can use \cref{thm:algorithm2} to test whether or not $\Ech_\sigma(x)=\Ech_{\sigma^\#}(x)$ for every linear extension $\sigma$ of $R$; this requires testing whether $\Ech_{\xi_1}(x)=\Ech_{\xi_2}(x)=\Ech_{\xi_3}(x)=\Ech_{\xi_4}(x)=y$ for just four specific linear extensions $\xi_1,\xi_2,\xi_3,\xi_4$. By \eqref{eq:ranks}, this amounts to computing the ranks of $16$ matrices ($4$ for each linear extension), each of which has at most $n$ rows and at most $n$ columns. Overall, the entire algorithm requires computing echelonmotion with respect to the initial linear extension $\sigma^\#$ and then computing the ranks of at most $16n$ matrices, each of which has at most $n$ rows and at most $n$ columns. This is substantially more efficient than computing echelonmotion separately with respect to all linear extensions of $R$.

\begin{remark}
Using the above algorithm, we have checked that the (strong) Bruhat order on the symmetric group $S_n$ is echelon-independent when $n\leq 5$ but not when $n=6$. It is well known that the MacNeille completion of the Bruhat order on $S_n$ is distributive (it is isomorphic to the standard order on alternating sign matrices \cite{LS}), so the Bruhat order on $S_6$ is another counterexample to the converse of \cref{thm:MacNeille}. 

Let us give more details of our computation. We write permutations in one-line notation, so $[w_1 w_2 \cdots w_n]$ means $1 \mapsto w_1$, $2 \mapsto w_2$, etc. \texttt{SAGE}'s default order on permutations, denoted $\sigma_{\texttt{SAGE}}$, is lexicographic order on these one-line notations. Let $x=[2 4 1 6 3 5]$ and $y=[5 1 3 2 6 4]$.  We have $\Ech_{\sigma_{\texttt{SAGE}}}(x) = y$. 
The elements $x$ and $y$ are incomparable in Bruhat order. If we choose an order $\sigma_1 \in \Xi_1(x,y)$, then we have $\Ech_{\sigma_{1}}(x) = [315462] \neq y$. 
\end{remark} 

\section{Semidistributive Lattices}\label{sec:semidistributive} 

Let $L$ be a semidistributive lattice. There is a canonical labeling of the edges of the Hasse diagram of $L$ with join-irreducible elements of $L$. The label of the edge $x\lessdot y$, denoted $j_{x,y}$, is the minimum element of the set $\{z\in L:z\vee x=y\}$. For each element $w\in L$, we define the \dfn{downward label set} $\DD_L(w)$ and the \dfn{upward label set} $\UU_L(w)$ by 
\[\DD_L(w)=\{j_{x,w}:x\lessdot w\}\quad\text{and}\quad \UU_L(w)=\{j_{w,y}:w\lessdot y\}.\] It is known that any two distinct elements of $L$ have different downward label sets and different upward label sets. Moreover, the sets $\{\DD_L(w):w\in L\}$ and $\{\UU_L(w):w\in L\}$ are equal and form a flag simplicial complex called the \dfn{canonical join complex} of $L$. Barnard \cite{B19} defined \dfn{rowmotion} to be the unique bijection $\row_L\colon L\to L$ such that 
\[\UU_L(\row_L(w))=\DD_L(w)\] for all $w\in L$; see \cref{fig:semidistributive_rowmotion}. Because $L$ is meet-semidistributive, the set \[\Odown(x)=\{z\in L:z\wedge x=\popdown_L(x)\}\] has a maximum element; Defant and Williams \cite{DW23} found that this maximum element is $\row_L(x)$. In fact, they showed that \begin{equation}\label{eq:pop=row}
\max\Odown(x)=\{\row_L(x)\}.
\end{equation} 
This alternative characterization of rowmotion will allow us to prove \cref{semidist-thm}. 

\begin{figure}[htbp]
  \begin{center}
  \includegraphics[height=5.2cm]{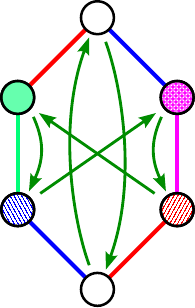}
  \end{center}
\caption{A semidistributive lattice. Each join-irreducible element has its own color, which is also used to color the edges that it labels. A green arrow is drawn from each element to its image under rowmotion.  }\label{fig:semidistributive_rowmotion}
\end{figure}

\begin{proof}[Proof of~\cref{semidist-thm}] 
Fix a linear extension $\sigma$ of a semidistributive lattice $L$. Since semidistributive lattices are semidistrim, we know by \cref{lem:Mobius_nonzero} that $\mu_L(\popdown_L(x),x)\neq 0$ for every $x\in L$. In addition, \eqref{eq:pop=row} tells us that 
\[\textstyle{\max_\sigma} (\Odown_L(x))=\row_L(x)\] for every $x\in L$. Since rowmotion is a bijection, it follows from \cref{cor:alpha} that $\Ech_\sigma=\row_L$.  

The fact that every echelon-independent lattice is semidistributive follows from \cref{thm:MacNeille}.
\end{proof} 

\section{Trim Lattices}\label{sec:trim}   

Our goal in this section is to define vertebral linear extensions of trim lattices and to prove \cref{thm:trim}. Let us start by collecting several useful notions and results regarding trim lattices, many of which come from \cite{T06,TW19,DW23}. \cref{fig:trim} illustrates many of the definitions. 

\begin{lemma}[{\cite[Theorem~1]{T06}}]
Every interval in a trim lattice is a trim lattice. 
\end{lemma}

Let $L$ be a trim lattice, and let $k=|\JJ_L|=|\MM_L|$. Because $L$ is extremal, its maximum-length chains have cardinality $k+1$. The \dfn{spine} of $L$ is the union of the maximum-length chains of $L$. Let $\C=\{\hat 0=u_0\lessdot u_1\lessdot u_2\lessdot\cdots\lessdot u_k=\hat 1\}$ be a maximum-length chain of $L$. For each $i\in[k]$, there is a unique join-irreducible element $j_i\in\JJ_L$ such that $j_i\vee u_{i-1}=u_i$, and there is a unique meet-irreducible element $m_i\in\MM_L$ such that $m_i\wedge u_i=u_{i-1}$. This yields a bijection $\kappa_L\colon \JJ_L\to\MM_L$ defined so that $\kappa_L(j_i)=m_i$ for all $i\in[k]$. For each cover relation $x\lessdot y$ in $L$, let \begin{equation}\label{eq:gamma}\gamma_{\C}(x\lessdot y)=\min\{i\in[k]:u_i\vee x\geq y\},\end{equation}
and label the edge between $x$ and $y$ with the join-irreducible $\jj_{x,y}=j_{\gamma(x\lessdot y)}$. (See \cref{fig:trim}.) 
According to \cite[Proposition~3.8]{TW19}, the edge label $\jj_{x,y}$ does not depend on the choice of the maximum-length chain $\C$. In addition, $\jj_{j_*,j}=\jj_{\kappa_L(j),\kappa_L(j)^*}=j$ for every $j\in\JJ_L$, so the bijection $\kappa_L$ is also independent of the choice of $\C$. Given an interval $L'$ of $L$, define $\Gamma_\C(L')$ to be the set of integers $\gamma_\C(x\lessdot y)$ such that $x\lessdot y$ is a cover relation in $L'$. 

The next result is a consequence of \cite[Propositions~1~\&~5]{T06}, which build off of results from~\cite{MT}. 

\begin{lemma}[{\cite[Propositions~1~\&~5]{T06}}]\label{prop:1and5}
Let $\C$ be a maximum-length chain of a trim lattice $L$. Let $L'=[v,w]$ be an interval in $L$, viewed as a trim lattice. The set $\C'=\{(v\vee u)\wedge w:u\in\C\}$ is a maximum-length chain of $L'$. Let $k'=|\JJ_{L'}|=|\MM_{L'}|=|\C'|-1$. Then $|\Gamma_{\C}(L')|=k'$. For each cover relation $x\lessdot y$ in $L'$, we have $\gamma_{\C'}(x\lessdot y)=\varphi(\gamma_\C(x\lessdot y))$, where $\varphi$ is the unique order-preserving bijection from $\Gamma_\C(L')$ to $[k']$. 
\end{lemma} 

For an element $u$ in a trim lattice $L$, define 
\[\DD_L(u)=\{\jj_{x,u}:x\in\Cov_L^{\downarrow}(u)\}\quad\text{and}\quad \UU_L(u)=\{\jj_{u,y}:y\in\Cov_L^\uparrow(u)\}.\] 

\begin{lemma}[{\cite[Theorem~5.6]{DW23}}]\label{lem:join_and_meet} 
Let $L$ be a trim lattice. For each $x\in L$, we have 
\[x=\bigvee\DD_L(x)=\bigwedge\kappa_L(\UU_L(x)).\] 
\end{lemma}  

The bijection $\kappa_L$ allows us to define a directed graph $G_L$ called the \dfn{Galois graph} of $L$. The vertex set of $G_L$ is $\JJ_L$. For $j,j'\in\JJ_L$, we have an arrow $j\to j'$ in $G_L$ if and only if $j\neq j'$ and $j\not\leq\kappa_L(j')$. (See \cref{fig:trim}.) The Galois graph of $L$ is acyclic \cite{TW19,DW23}. 

\begin{lemma}[{\cite[Corollary~5.6]{TW19}}]\label{lem:independent_sets} 
Let $L$ be a trim lattice. Each of the maps $\DD_L$ and $\UU_L$ is a bijection from $L$ to the collection of independent sets of $G_L$. 
\end{lemma} 

Let $\C$ be a maximum-length chain of a trim lattice $L$. For $u\in L$, define $\w_{\C}(u)$ to be the word over the alphabet $[k+1]$ obtained by listing the elements of \[\{\gamma_{\C}(u\lessdot y):y\in\Cov_L^\uparrow(u)\}\cup\{k+1\}\] in increasing order. Because the map $\UU_L$ is injective on $L$, the map $\w_{\C}$ is also injective on $L$. Let $\leq_\lex$ denote the lexicographic total order on the set of words over the alphabet $\mathbb Z$. Let $\sigma_{\C}\colon L\to[n]$ be the unique bijection such that 
\[\w_{\C}(\sigma_\C^{-1}(1))<_\lex \w_{\C}(\sigma_\C^{-1}(2))<_\lex\cdots<_\lex\w_{\C}(\sigma_\C^{-1}(n)).\] 
(See \cref{fig:trim}.) 

\begin{lemma}\label{lem:CC'} 
Let $L$ be a trim lattice with maximum element $\hat 1$, and let $v\in L$. Let $L'=[v,\hat 1]$, and let $n'=|L'|$. Let $\C$ be a maximum-length chain of $L$, and let $\C'=\{v\vee u:u\in\C\}$. Then $\sigma_{\C'}=\psi\circ\sigma_\C$, where $\psi$ is the unique order-preserving bijection from $\sigma_\C(L')$ to $[n']$. 
\end{lemma}
\begin{proof}
Let $k=|\JJ_L|$ and $k'=|\JJ_{L'}|$. According to \cref{prop:1and5}, we have $|\Gamma_{\C}(L')|=k'$, so there is a unique order-preserving bijection $\varphi\colon\Gamma_\C(L')\to[k']$. For each $u\in L'$, we have $\Cov_{L'}^\uparrow(u)=\Cov_{L}^\uparrow(u)$, so it follows from \cref{prop:1and5} that 
\[\{\gamma_{\C'}(u\lessdot y):y\in\Cov_{L'}^\uparrow(u)\}=\{\varphi(\gamma_\C(u\lessdot y)):y\in\Cov_L^\uparrow(u)\}.\] Therefore, the word $\w_{\C'}(u)$ is obtained from $\w_\C(u)$ by replacing the letter $k+1$ with $k'+1$ and applying $\varphi$ to all of the other letters. This implies the desired result. 
\end{proof}

\begin{lemma}\label{lem:vertebral_are_linear} 
Let $\C$ be a maximum-length chain of a trim lattice $L$. Then $\sigma_{\C}$ is a linear extension of $L$. 
\end{lemma} 

\begin{proof}
We proceed by induction on $|L|$. Let $\hat 0$ and $\hat 1$ be the minimum and maximum elements of $L$, respectively. Choose $z,z'\in L$ with $z\leq z'$; we wish to show that $\w_\C(z)\leq_\lex \w_\C(z')$ (equivalently, $\sigma_\C(z)\leq\sigma_\C(z')$). We consider two cases. 

\medskip 

\noindent {\bf Case 1.} Suppose $z\neq\hat 0$. Let $L'=[z,\hat 1]$, and let $\C'=\{z\vee u:u\in\C\}$. We know by \cref{prop:1and5} that $\C'$ is a maximum-length chain of $L'$. Since $|L'|<|L|$, we know by induction that ${\sigma_{\C'}(z)\leq \sigma_{\C'}(z')}$. It follows from \cref{lem:CC'} that $\sigma_\C(z)\leq \sigma_\C(z')$, as desired. 

\medskip 

\noindent {\bf Case 2.} Suppose $z=\hat 0$. Let $j_1$ be the atom in $\C$, and let $m_1=\kappa_L(j_1)$. Let $L'=[\hat 0,m_1]$ and $L''=[j_1,\hat 1]$. Let $k=|\JJ_L|$ and $k'=|\JJ_{L'}|$. According to \cite[Lemma~3.12]{TW19}, we have the decomposition $L=L'\sqcup L''$. According to \cite[Lemma~3.14]{TW19}, an element $w\in L$ is in $L'$ if and only if $j_1\in\UU_L(w)$. In other words, we have $w\in L'$ if and only if the letter $1$ appears in $\w_\C(w)$. \cref{prop:1and5} tells us that $k'=|\Gamma_\C(L')|$. Let $\varphi$ be the order-preserving bijection from $\Gamma_\C(L')$ to $[k']$. 
For each $w\in L'$, it follows from \cite[Lemma~3.15]{TW19} that the unique element of $\Cov_L^\uparrow(w)\cap[j_1,\hat 1]$ is $w\vee j_1$, so $\Cov_{L'}^\uparrow(w)=\Cov_L^\uparrow(w)\setminus\{w\vee j_1\}$. Therefore, it follows from \cref{prop:1and5} that $\w_{\C'}(w)$ is obtained from $\w_\C(w)$ by deleting the letter $1$, changing the letter $k+1$ to $k'+1$, and applying $\varphi$ to all of the other letters. For each $w\in L'$, we know by induction (since $|L'|<|L|$) that $\w_{\C'}(\hat 0)\leq_\lex\w_{\C'}(w)$, so we must also have $\w_\C(\hat 0)\leq_\lex\w_\C(w)$. 

If $z'\in L'$, then we have just seen that $\w_\C(\hat 0)\leq_\lex\w_\C(z')$, as desired. On the other hand, if $z'\in L''$, then $\w_\C(\hat 0)<_\lex\w_\C(z')$ because $\w_\C(\hat 0)$ contains the letter $1$ while $\w_\C(z')$ does not.  
\end{proof} 

\begin{definition}\label{def:vertebral} 
We say a linear extension of a trim lattice $L$ is \dfn{vertebral} if it is of the form $\sigma_\C$ for some maximum-length chain $\C$ of $L$. 
\end{definition}

Given a trim lattice $L$, we follow \cite{TW19} and define \dfn{rowmotion} to be the bijection $\row_L\colon L\to L$ characterized by the equality 
\[\DD_L=\UU_L\circ\row_L.\] 
Note that the restriction of $\row_L$ to $\JJ_L$ is $\kappa_L$. Defant and Williams proved \cite[Theorem~9.3]{DW23} that \begin{equation}\label{eq:trim_pop=row}
\row_L(x)\in\max\Odown_L(x)
\end{equation} for every $x\in L$. 
(See \cref{fig:trim}.)  

\begin{figure}[htbp]
  \begin{center}
  \includegraphics[height=6.370cm]{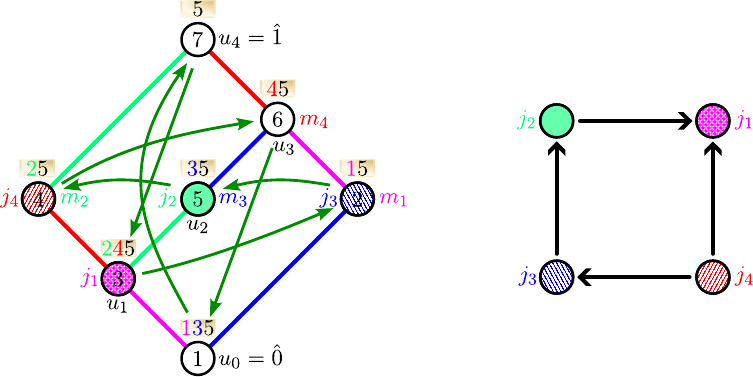}
  \end{center}
\caption{On the left is a trim lattice with the elements of a maximum-length chain $\C=\{u_0,u_1,u_2,u_3,u_4\}$ labeled. Each element $x$ is represented by a circle filled with $\sigma_\C(x)$. We have also labeled the join-irreducible elements $j_1,j_2,j_3,j_4$ and the meet-irreducible elements $m_1,m_2,m_3,m_4$. Each join-irreducible element has its own color, which is also used to color the edges that it labels. Above each element $x$ is a rectangle containing the word $\w_\C(x)$. A green arrow is drawn from each element to its image under rowmotion. On the right is the Galois graph of this lattice. }\label{fig:trim}
\end{figure}

\begin{example} 
\cref{def:vertebral} is illustrated in \cref{fig:trim}. The lattice $L$ on the left is trim but not semidistributive. 
We have chosen a maximum-length chain $\C=\{\hat{0}\lessdot u_1\lessdot u_2\lessdot u_3\lessdot u_4=\hat 1\}$. Each element $x$ is represented by a circle filled with $\sigma_\C(x)$, and above that circle is a golden rectangle containing the word $\w_\C(x)$. Note that 
\[{\color{Pink}1}{\color{blue}3}5<_\lex {\color{Pink}1}5<_\lex {\color{BlueGreen}2}{\color{red}4}5<_\lex {\color{BlueGreen}2}5<_\lex {\color{blue}3}5<_\lex {\color{red}4}5<_\lex 5.\] 
Rowmotion on $L$ is indicated by green arrows. 
\end{example}

We can now prove that echelonmotion on a trim lattice with respect to a vertebral linear extension coincides with rowmotion. 

\begin{proof}[Proof of~\cref{thm:trim}] 
Let $\sigma$ be a vertebral linear extension of a trim lattice $L$. Thus, we have $\sigma=\sigma_\C$ for some maximum-length chain $\C$ of $L$. We know by \cref{lem:Mobius_nonzero} that $\mu_L(\popdown_L(x),x)\neq 0$ for every $x\in L$. We will prove that \begin{equation}\label{eq:trim_goal_1}
\textstyle {\max_\sigma}(\Odown_L(x))=\row_L(x)
\end{equation} 
for every $x\in L$. Since rowmotion is a bijection, it will then follow from \cref{cor:alpha} that $\Ech_\sigma=\row_L$. Our strategy is to proceed by induction on $|L|$. 

Let $x\in L$. If $x=\hat 0$, then $\row(x)=\hat 1$, so \eqref{eq:trim_goal_1} certainly holds. Thus, we may assume $x>\hat 0$ and proceed by induction on the lattice $L$. In other words, we may assume ${\max_\sigma(\Odown_L(\widetilde x))=\row_L(\widetilde x)}$ for all $\widetilde x\in L$ such that $\widetilde x<x$. 

Let $L'=[\popdown_L(x),\hat 1]$. The set $\C'=\{\popdown_L(x)\vee u:u\in\C\}$ is a maximum-length chain of $L'$ by \cref{prop:1and5}. We have $\popdown_{L'}(x)=\popdown_L(x)$, so $\Odown_{L'}(x)=\Odown_L (x)$. It follows from \cref{prop:1and5} that $\row_{L'}(x)=\row_L(x)$. If $|L'|<|L|$, then we know by induction that 
\[\textstyle {\max_{\sigma_{\C'}}}(\Odown_{L'}(x))=\row_{L'}(x),\] so \[\textstyle {\max_{\sigma_{\C'}}}(\Odown_{L}(x))=\row_{L}(x).\] Therefore, if $|L'|<|L|$, then \eqref{eq:trim_goal_1} follows from \cref{lem:CC'}. Hence, we may assume in what follows that $L=L'$. In other words, we will assume that $\popdown_L(x)=\hat 0$. 

Let $y\in\Odown_L(x)$; we will show that \begin{equation}\label{eq:wCy} 
\w_\C(y)\leq_\lex\w_\C(\row_L(x)).
\end{equation} As $y$ is arbitrary, this will prove \eqref{eq:trim_goal_1}.

Let $\C=\{\hat 0=u_0\lessdot u_1\lessdot u_2\lessdot\cdots\lessdot u_k=\hat 1\}$. For $i\in[k]$, let $j_i$ be the unique join-irreducible element of $L$ such that $j_i\vee u_{i-1}=u_i$. Let $r$ be the smallest integer such that $j_r$ belongs to the set $\DD_L(x)=\UU_L(\row_L(x))$. In other words, $r$ is the first letter in $\w_\C(\row_L(x))$. According to \cite[Proposition~9.5]{DW23}, we have $\DD_L(x) \subseteq \UU_L(\popdown_L(x))$. Since $\popdown_L(x)=\hat 0$, there exists an atom $a$ of $L$ such that $\jj_{\hat 0,a}=j_r$. But $a$ is join-irreducible, so $\jj_{\hat 0,a}=a$. Hence, $a=j_r$. Because $j_r\in\DD_L(x)$, we know by \cref{lem:join_and_meet} that $j_r\leq x$. Now, $y\wedge x=\popdown_L(x)=\hat 0$, so $j_r\not\leq y$. Since $j_r\leq u_r$, this implies that $y<u_r\vee y$, so there exists $z$ such that $y\lessdot z\leq u_r\vee y$. Let $r'=\gamma_{\C}(y\lessdot z)$ so that $\jj_{y,z}=j_{r'}$. By the definition of $\gamma_\C(y\lessdot z)$ (see \eqref{eq:gamma}), we know that $r'\leq r$. The letter $r'$ appears in $\w_\C(y)$. Therefore, if $r'<r$, then $\w_\C(y)<_\lex\w_\C(\row_L(x))$, as desired. Hence, we may assume in what follows that $r'=r$. 

We have $j_r\in\DD_L(x)=\UU_L(\row_L(x))$ and $j_r\in\UU_L(y)$. It follows from \cref{lem:independent_sets} that there exist $x',y'\in L$ such that \[\DD_L(x')=\DD_L(x)\setminus\{j_r\}\quad\text{and}\quad\UU_L(y')=\UU_L(y)\setminus\{j_r\}.\] Note that $\UU_L(\row_L(x'))=\UU_L(\row_L(x))\setminus\{j_r\}$. The words $\w_\C(\row_L(x'))$ and $\w_\C(y')$ are obtained from $\w_\C(\row_L(x))$ and $\w_C(y)$, respectively, by deleting the letter $r$. Hence, in order to prove \eqref{eq:wCy}, it suffices to show that \begin{equation}\label{eq:wCy'}
\w_\C(y')\leq_\lex\w_\C(\row_L(x')).
\end{equation} 
According to \cref{lem:join_and_meet}, we have 
\[x'=\bigvee\DD_L(x')\leq\bigvee\DD_L(x)=x.\] We must have $x'<x$ since $\DD_L(x)\neq\DD_L(x')$. Therefore, we know by induction that \[
\textstyle {\max_\sigma}(\Odown_L(x'))=\row_L(x').\] In other words, we have $\w_\C(z)\leq\w_\C(\row_L(x'))$ for every $z\in\Odown_L(x')$. Thus, to complete the proof of \eqref{eq:wCy'} (and, therefore, the proof of the theorem), we just need to show that $y'\in\Odown_L(x')$. 

For each $j\in\DD_L(x')$, there is no arrow from $j$ to $j_r$ in the Galois graph $G_L$ because $j$ and $j_r$ both belong to the independent set $\DD_L(x)$. By the definition of the Galois graph, this means that $j\leq\kappa_L(j_r)$. As this is true for all $j\in\DD_L(x')$, we can use \cref{lem:join_and_meet} to deduce that 
\[x'=\bigvee\DD_L(x')\leq\kappa_L(j_r).\] Hence, $x'=x'\wedge\kappa_L(j_r)$. Appealing to \cref{lem:join_and_meet} again and using the fact that $x'\leq x$, we find that 
\begin{align*}
x'\wedge y'&=x'\wedge\bigwedge\kappa_L(\UU_L(y')) \\ 
&=(x'\wedge\kappa_L(j_r))\wedge\bigwedge\left(\kappa_L(\UU_L(y))\setminus\{\kappa_L(j_r)\}\right) \\ 
&=x'\wedge\bigwedge\kappa_L(\UU_L(y)) \\ 
&=x'\wedge y \\
&\leq x\wedge y \\ 
&=\hat{0}, 
\end{align*}
where the last equality holds because $y\in\Odown_L(x)$ and $\popdown_L(x)=\hat{0}$. 
Similarly, 
\begin{align*}
x'\wedge \row_L(x')&=x'\wedge\bigwedge\kappa_L(\UU_L(\row_L(x'))) \\ 
&=(x'\wedge\kappa_L(j_r))\wedge\bigwedge\left(\kappa_L(\UU_L(\row_L(x)))\setminus\{\kappa_L(j_r)\}\right) \\ 
&=x'\wedge\bigwedge\kappa_L(\UU_L(\row_L(x))) \\ 
&=x'\wedge \row_L(x) \\
&\leq x\wedge \row_L(x) \\ 
&=\hat{0}. 
\end{align*} 
This shows that 
\[x'\wedge y'=x'\wedge\row_L(x')=\hat{0}.\] Since $\row_L(x')\in\Odown_L(x')$ by \eqref{eq:trim_pop=row}, we have $\popdown_L(x')=x'\wedge\row_L(x')=\hat{0}$. This shows that $x'\wedge y'=\hat{0}=\popdown_L(x')$, so $y'\in\Odown_L(x')$, as desired. 
\end{proof}

\section{Eulerian Posets}\label{sec:Eulerian} 

Recall that a graded poset $R$ is \emph{Eulerian} if $\mu_R(x,y)=(-1)^{\mathrm{rk}(y)-\mathrm{rk}(x)}$ for every interval $[x,y]$ in $R$. We now prove that echelonmotion on an Eulerian poset (with respect to any linear extension) is an involution. 

\begin{proof}[Proof of~\cref{thm:Eulerian}] 
Let $\sigma$ be a linear extension of an $n$-element Eulerian poset $R$. Let $W=W^{R,\sigma}$ be the associated Cartan matrix. It follows from the definition of the M\"obius function that the entries of the inverse matrix $W^{-1}$ are given by 
\[W^{-1}_{i,j}=\begin{cases}
    \mu_R(\sigma^{-1}(j),\sigma^{-1}(i)) & \text{if }\sigma^{-1}(i)\geq \sigma^{-1}(j) \\
    0 & \text{if }\sigma^{-1}(i)\not\geq \sigma^{-1}(j). 
\end{cases}\] Because $R$ is Eulerian, 
\[W^{-1}_{i,j}=\begin{cases}
    (-1)^{\mathrm{rk}(\sigma^{-1}(i))-\mathrm{rk}(\sigma^{-1}(j))} & \text{if }\sigma^{-1}(i)\geq \sigma^{-1}(j) \\
    0 & \text{if }\sigma^{-1}(i)\not\geq \sigma^{-1}(j),  
\end{cases}\]
where $\mathrm{rk}$ is the rank function of $R$. 
Let $D$ be the $n\times n$ diagonal matrix whose $i$-th diagonal entry is $(-1)^{\mathrm{rk}(\sigma^{-1}(i))}$. Then 
$W^{-1}=DWD^{-1}$. Let $P=P^{R,\sigma}$ be the permutation matrix appearing in the Bruhat decomposition of $W$. This means that there are upper-triangular matrices $U_1$ and $U_2$ such that $W=U_1PU_2$. But then $W^{-1}=DU_1PU_2D$. The matrices $DU_1$ and $U_2D$ are upper-triangular, so the permutation matrix in the Bruhat decomposition of $W^{-1}$ is $P$. However, we can also write $W^{-1}=U_2^{-1}P^{-1}U_1^{-1}$ to see that the permutation matrix in the Bruhat decomposition of $W^{-1}$ is $P^{-1}$. It follows that $P=P^{-1}$, so $\Ech_\sigma$ is an involution. 
\end{proof} 

\section{Concluding Remarks}\label{sec:conclusion}  

We believe that echelonmotion deserves much more attention, so we conclude with some problems and questions aimed at guiding future research.

\subsection{Echelon-Independence and Auslander-Regular Posets} 
Recall that \cref{thm:bounded,thm:MacNeille,semidist-thm} make progress toward the following problem. 

\begin{problem}
Classify echelon-independent posets. 
\end{problem} 

Define an equivalence relation $\equiv_{\Ech}$ on the set of linear extensions of a poset $R$ by declaring that $\sigma\equiv_{\Ech}\sigma'$ if and only if $\Ech_\sigma=\Ech_{\sigma'}$. Thus, $R$ is echelon-independent if and only if $\equiv_{\Ech}$ has exactly one equivalence class. Even when $R$ is not echelon-independent, it could be interesting to study properties of this equivalence relation (such as the number of equivalence classes). 

Let us briefly discuss an interesting connection between the rowmotion operator and homological algebra of finite-dimensional algebras.
Let $K$ be a field, and let $A$ be a finite-dimensional $K$-algebra of finite global dimension.
Then $A$ is called \dfn{Auslander-regular} if the minimal injective coresolution
$$0 \rightarrow A \rightarrow I^0 \rightarrow I^1 \rightarrow I^2 \rightarrow \cdots \rightarrow I^n \rightarrow 0$$
has the property that the projective dimension of $I^i$ is at most $i$ for all $i \geq 0$. 
This is the non-commutative generalization of the classical commutative regular rings due to Auslander; see the survey \cite{Cl}.
Iyama and Marczinzik \cite{IM} proved that a finite lattice $L$ is distributive if and only if the incidence algebra of $L$ is Auslander-regular.
Every Auslander-regular algebra $A$ comes with a natural permutation $\pi$, called the \dfn{Auslander--Reiten permutation}, which is defined by the condition that the indecomposable projective module $P_{\pi(i)}$ is the last term in the minimal projective resolution of the indecomposable injective module $I_i$.
For other equivalent descriptions of the Auslander--Reiten permutation, see \cite{KMT25}.
It was shown in \cite{IM} that the Auslander--Reiten permutation on the incidence algebra of a distributive lattice coincides with the inverse of the rowmotion operator on the elements of this distributive lattice.
We call a finite poset $R$ \dfn{Auslander-regular} if the incidence algebra $KR$ is Auslander-regular when $K$ has characteristic $0$. 
The classification of Auslander-regular posets is an open problem with some partial progress in the forthcoming article \cite{IKKM}.
Examples of Auslander-regular posets that are not lattices include the strong Bruhat order on the symmetric groups $S_3$ and $S_4$. Many more examples related to number theory will be given in \cite{IKKM}.
We have the following conjecture:

\begin{conjecture} 
Let $R$ be a connected Auslander-regular poset. Then $R$ is echelon-independent, and echelonmotion on $R$ is the inverse of the Auslander--Reiten permutation of $R$. 
\end{conjecture}
We have verified this conjecture for posets with at most 10 elements. Note that a connected Auslander-regular poset is bounded and has a distributive MacNeille completion; see \cite{IKKM}.
As a final remark, we note that we can associate to every finite-dimensional algebra of finite global dimension an echelonmotion using the fact that the Cartan matrix is invertible in that case. Echelonmotion on a poset with respect to a linear extension $\sigma$ is then obtained as a special case by considering the Cartan matrix of the incidence algebra of the poset with respect to $\sigma$. 
It turns out that this definition is also interesting for other finite-dimensional algebras associated to combinatorial objects such as Dyck paths, which correspond bijectively to linear Nakayama algebras; see \cite{KKM}.

\subsection{Canonical Bruhat Factorizations} 

Let $\sigma$ be a linear extension of a lattice $L$, and let $W=W^{L,\sigma}$ be the corresponding Cartan matrix. Recall that the \emph{Coxeter matrix} of $L$ with respect to $\sigma$ is defined to be $C=-W^{-1}W^\top$. Kl\'asz, Marczinzik, and Thomas \cite{KMT25} proved that $L$ is distributive if and only if there exist an upper-triangular matrix $U$ and a permutation matrix $P$ such that $C=PU$. In other words, $L$ is distributive if and only if $C$ has a Bruhat decomposition in which the first upper-triangular matrix is the identity matrix. Furthermore, the matrix $U$ in this case is an involution \cite{MTY24}. One can view this factorization as a ``canonical Bruhat factorization'' of the Coxeter matrix of a distributive lattice. This leads us to the following (somewhat vague) question.  

\begin{question}
Let $\sigma$ be a linear extension of a semidistributive lattice $L$. Is there a canonical Bruhat factorization of the Coxeter matrix of $L$ with respect to $\sigma$? 
\end{question} 

\subsection{Independence Posets} 
In \cref{sec:trim}, we defined vertebral linear extensions of trim lattices and proved that echelonmotion with respect to a vertebral linear extension agrees with rowmotion. Thomas and Williams \cite{TW20} introduced \emph{independence posets} as a generalization of trim lattices, and they explained how to define a bijective rowmotion operator on an independence poset. It is possible to define vertebral linear extensions of independence posets in a manner that directly mimics \cref{def:vertebral}; with this definition, every independence poset has at least one vertebral linear extension. This leads us naturally to the following question. 

\begin{question}
Let $\sigma$ be a vertebral linear extension of an independence poset $R$. Is $\Ech_\sigma$ the same as rowmotion on $R$? 
\end{question} 

\subsection{Modular Lattices} 
A lattice $L$ is \dfn{modular} if for all $a,b,x\in L$ such that $a\leq b$, we have $a\vee(x\wedge b)=(a\vee x)\wedge b$. Every distributive lattice is modular. There are also numerous other notable families of modular lattices. For instance, the lattice of normal subgroups of a finite group is modular. As another example, the lattice of submodules of a module over a ring is modular. 

In a seminal article, Dilworth \cite{Dilworth} proved that for every modular lattice $L$ and every nonnegative integer $k$, the number of elements of $L$ that are covered by $k$ elements equals the number of elements of $L$ that cover $k$ elements. That is, 
\[\left|\{x\in L:|\Cov_L^\uparrow(x)|=k\}\right|=\left|\{x\in L:|\Cov_L^\downarrow(x)|=k\}\right|.\]
The following conjecture states that for any linear extension $\sigma$ of $L$, echelonmotion with respect to $\sigma$ explicitly realizes Dilworth's theorem. 

\begin{conjecture}\label{conj:modular} 
Let $\sigma$ be a linear extension of a modular lattice $L$. For every $x\in L$, we have 
\[\left|\Cov_L^{\uparrow}(\Ech_\sigma(x))\right|=\left|\Cov_L^{\downarrow}(x)\right|.\] 
\end{conjecture} 

We have tested \cref{conj:modular} for all modular lattices with at most $9$ elements. For each prime power $q\leq 4$ and each positive integer $d\leq 3$, we have tested the conjecture for 200 random linear extensions of the lattice of subspaces of $\mathbb{F}_q^d$. 

\subsection{Beyond Linear Extensions} 
Let $R$ be an $n$-element poset. The article \cite{IM} considers bijections $\sigma\colon R\to[n]$ that need not be linear extensions. Given any such bijection, one can define the Cartan matrix $W^{R,\sigma}$ and the echelonmotion map $\Ech_\sigma\colon R\to R$ just as in \cref{def:echelonmotion}. While we have focused exclusively on linear extensions, it could be fruitful to explore this more general setting.

\section*{Acknowledgments} 
This work began at the 2025 workshop on Lattice Theory at the Banff International Research Station. We thank BIRS for the excellent working conditions that they provided. We thank the other participants of the workshop for their insights. We are particularly grateful to Emily Barnard, Cesar Ceballos, and Osamu Iyama for co-organizing the workshop and to Elise Catania, Cl\'ement Chenevi\`ere, Jessica Striker, Christian Stump, and Emine Y\i{}ld\i{}r\i{}m for several helpful conversations. 

Colin Defant was supported by the National Science Foundation under Award No.\ 2201907 and by a Benjamin Peirce Fellowship at Harvard University. Adrien Segovia was supported by NSERC Discovery Grants RGPIN-2022-03960 and RGPIN-2024-04465. 
David Speyer was partially supported by the National Science Foundation under Award No.\ 2246570.
Hugh Thomas was partially supported by NSERC Discovery Grant RGPIN-2022-03960 and the Canada Research Chairs program, grant number CRC-2021-00120. Nathan Williams was partially supported by the National Science Foundation under Award No.\ 2246877.

\bibliographystyle{alpha}
\bibliography{ref}

\end{document}